\documentclass{amsart}

\usepackage{palatino, mathpazo}
\usepackage[mathscr]{eucal}
\usepackage{mathrsfs}

\usepackage{amsfonts,amssymb,amsmath}
\usepackage{epic, eepic}

\usepackage[all]{xy}

\usepackage{hyperref}

\newtheorem{thm}{Theorem}[section]
\newtheorem{lem}[thm]{Lemma}

\newtheorem{defn}[thm]{Definition}

\theoremstyle{remark}

\newtheorem{remark}[thm]{Remark}

\newcommand{\bC}{\mathbb{C}}

\newcommand{\bP}{\mathbb{P}}
\newcommand{\bQ}{\mathbb{Q}}

\newcommand{\cM}{\mathcal{M}}
\newcommand{\cO}{\mathcal{O}}

\newcommand{\mbar}{\overline{\cM}}

\newcommand{\on}{\operatorname}
\newcommand{\vir}{\textrm{vir}}

\DeclareMathOperator{\id}{Id}

\DeclareMathOperator{\spec}{Spec}

\begin{document}

\title{A remark on virtual pushforward properties in Gromov-Witten theory}

\begin{abstract}
We approach Gathmann's virtual pushforward property from the perspective of bivariant intersection theory,
extend a virtual pushforward result of Manolache, and use our extension to deduce
a result of Gathmann 
relating relative and rubber GW invariants of a $\mathbb{P}^1$ bundle with invariants of its base.
\end{abstract}

\author{Feng~Qu}
\address{
Beijing International Center for Mathematical Research, 
Beijing, 100871, China.}
\email{fengquest@gmail.com}

\maketitle

\section{Introduction}
The purpose of this note is to relate
the virtual pushforward property defined by Gathmmann to degree zero operational Chow rings, then reformulate and extend results of Manolache
in \cite{vpf}.

A virtual pushforward property is defined in 
\cite[Definition 5.2.1]{Ga} as follows.
Let $f\colon F \to G$ be a proper map between moduli stacks of stable maps over $\bC$, 
$[F]^\vir$ and $[G]^\vir$
their virtual classes in Chow groups respectively. Assume $d=\deg[F]^\vir -\deg[G]^\vir$ is nonnegtive. $f$ is said
to have the virtual pushforward property
if
$
f_*(\gamma \cap [F]^\vir) 
$ is zero when $\gamma \in A^{< d}(F)$, and a scalar multiple of $[G]^\vir$ in $A_*(G)$ if $\gamma \in A^d(F)$.
Here $A_*, A^*$ denote Chow groups and Chow rings respectively, and $\gamma$ is made up of evaluation classes and cotangent line classes.
A key insight in  \cite{Lai,vpf} is that virtual pushforward property for $f$ becomes tractable if $f$ is virtually smooth, i.e., there exists a virtual
pullback (\cite{vpb}) $f^!$ such that $f^!([G]^\vir)=[F]^\vir$.

We note that it is more flexible to allow $f^!$ to be a bivariant class 
in $A^{-d}(F\xrightarrow{f}G)$.
If the bivariant class $f_*(\gamma\cdot f^!)$ in $A^{\le 0}(G)$ is a scalar multiplication, then
$f$ has the virtual pushforward property.
This formulation seems to simplify proofs and makes the underlying intersection theory arguments more transparent. 

Results in \cite{vpf} about a strong virtual pushforward property  can be reformulated using bivariant classes as follows.
It is straightforward to show $A^{<0}(G)=0$, and a bivariant class in $A^0(G)$ is determined by its action on $B_0(G)$, the group of zero cycles in $G$ modulo algebraic equivalences with rational coefficients.
A class is a scalar multiplication if it action on $B_0(G)$ is.
 In particular, if $B_0(G)=\bQ$, any $\bQ$-linear endomorphism of $B_0(G)$ is a scalar multiplication, and this implies $A^0(G)$ consists only of scalar multiplications, forcing $f$ to satisfy the virtual pushforward property
 (cf. \cite[Theorem 3.13]{vpf}).
More generally, a class in $A^0(G)$ is a scalar multiplication if it is a pullback via some map $h\colon G \to G'$ such that $B_0(G)=\bQ$ (cf. \cite[Corollary 3.15]{vpf}).

Any scalar multiplication in $A^0(G)$ is the pullback of a scalar multiplication via $G \to \spec \bC$, so we already have a complete characterization of scalar multiplications in $A^0(G)$. However, the requirement of being a pullback might be too restrictive to apply.
In this note, we show that
\begin{thm} \label{FFF}
Let $h \colon G \to G'$ be a map between proper DM stacks over $\bC$,  and $c \in A^0(G)$ and $c' \in A^0(G')$ two bivariant classes.
If the diagram
\[
\xymatrix{
A_0(G) \ar[d]^{h_*} & A_0(G)\ar[l]_c\ar[d]^{h_*}\\
A_0(G')          &   A_0(G')\ar[l]_{c'},
}
\] is commutative,
 then
$c$ is a scalar multiplication if $c'$ is.
In particular, if $B_0(G')=\bQ$, then $c$ is a scalar multiplication.
\end{thm}

Theorem \ref{FFF} has a rather trivial proof. To justify its utility,
we apply it to deduce \cite[Theorem 5.2.7]{Ga} relating relative and rubber GW invariants of a $\bP^1$ bundle with invariants of its base. (See Theorem \ref{bundle}.) The original proof is based on (relative) virtual localization with respect to fiberwise $\bC^*$ actions.

The paper grew out of our efforts
to understand arguments and ideas  in \cite{Ga,Lai,vpb,vpf}, to which 
it is a pleasure to acknowledge our intellectural debt.

\section{Notation and Conventions}

We work over $\bC$.
 All the stacks are DM stacks of finite type over $\bC$.

\subsection{}
For a DM stack $X$, $B_*(X)$ denote the group of algebraic equivalence classes. (See \cite[Definition 2.28]{vpf}, \cite[Definition 10.3]{Fu}.)
Chow groups $A_*(X)$ and $B_*(X)$ are of rational coefficients. 

For bivariant classes, we follow the notation and conventions in \cite{Fu}.
Recall a bivariant class $\alpha \in A^{m}(F \xrightarrow{f} G)$ consists of 
maps 
\[\alpha(i)\colon A_*(V) \to A_{*-m}(W)
\]
for each cartesian diagram
\[
\xymatrix{
W\ar[r]\ar[d]                 &   V\ar[d]^i\\
F \ar[r]^f                           &    G.
}
\]
These maps are required to commute with proper pushforwards, flat pullbacks, and Gysin pullbacks.
(See \cite[Chapter 17]{Fu}.) If follows from these commutativities that $\alpha$ also induces maps 
\[
B_*(V) \to B_{*-m}(W).
\]

In particular, a class $c \in A^0(G)=A^0(\id_G)$ consists of maps $A_*(F) \to A_*(F)$ for each $F \to G$.

We will simply denote $\alpha(i)$ by $\alpha$ when
the map $i$ is clear from the context. $A^{m}(F \xrightarrow{f} G)$ is shortened to $A^{m}(f)$.

\subsection{} \label{stablemaps}
Let $X$ be a smooth projective variety over $\bC$, $L$ a line bundle over $X$,
and $\pi: Y \to X$ be the projective bundle $\bP_X(L\oplus\cO) \to X$. 

$Y$ has two sections $Y_0, Y_\infty$, where $Y_0$ is the zero section of $L$, and $Y_\infty=\bP_X(L\oplus\cO)-L$.
We can consider relative stable maps into $Y$ relative to the divisors $Y_0, Y_\infty$, or $Y_0 \coprod Y_\infty$,
and rubber stable maps into $Y$ relative to $Y_0\coprod Y_\infty$.

Our convention concerning moduli stacks of stable maps is as follows.
We will use $\mbar_{\Gamma_X}(X)$ to denote a moduli stack of stable maps to $X$ with discrete data $\Gamma_X$, here $\Gamma_X$ specifies genera, marked points, and curve classes of each connected component of the source curve, and possibly redundant information about contact orders. For relative and rubber moduli stacks of $Y$, we use
$\mbar_{\Gamma_Y}(Y^\dagger)$ and $\mbar_{\Gamma_Y}(Y^\dagger)^\sim$ respectively.
$\Gamma_Y$ not only specifies genera, marked points, and curve classes of each connected component of the source curve,
but also records contact orders of each marked point relative to the divisor(s). Here $Y^\dagger$ was used to indicate that
our target is a log
scheme, and $\sim$ for rubbers. 

We use Kim's stable log maps as models for $\mbar_{\Gamma_Y}(Y^\dagger)$ 
(\cite{Ki}) and $\mbar_{\Gamma_Y}(Y^\dagger)^\sim$(\cite{MR}). They are compatible with the 
original version of J. Li (\cite{Li}), and Graber-Vakil (\cite{GV}) by results and methods of \cite{AMW}.

Given a map $f\colon Z \to U$, and 
some discrete data $\Gamma_Z$ for $Z$, we will use $f_*\Gamma_Z$ to denote the discrete data on $U$ obtained by replacing
curve classes by their images under $f$.

\section{}
We prove Theorem \ref{FFF} in this section.

\subsection{}
Let $G$ be a proper DM stack.
We first show that a class $c$ of $A^0(G)$ is determined by its action on
$B_0(G)$.

\begin{lem}\label{GGG}

Assume  $G$ is proper.
Let $c\in A^0(G)$, then $c$ is determined by $c\colon B_0(G) \to B_0(G)$, and it is a scalar multiplication
by a rational number $n$ if and only if $c\colon B_0(G) \to B_0(G)$ is.
\end{lem}
\begin{proof}

It is easy to see
$c$ is determined by maps $c\colon A_{\dim W}(W) \to A_{\dim W}(W)$ for integral (irreducible and reduced)
$W$ over $G$. In fact, for any $F \to G$, the map $c\colon A_*(F) \to A_*(F)$ is determined by its action on integral substacks of $F$. Let  $i\colon V \to F$ be a closed embedding of an integral DM stack $V$ with fundamental class $[V]$ into $F$. As $A_{\dim V}(V)=\bQ[V]$,
\[ 
c\colon A_{\dim V}(V) \to A_{\dim V}(V)
\]
is defined by 
\[
c([V])=n(V)[V]. 
\] for some rational number $n(V)$.
Then we see that $c(i_*[V])=n(V)(i_*[V])$ as $c$ commutes with the pushforward $i_*$.
Therefore $c\colon A_*(F) \to A_*(F)$ is determined by those $n(V)$ where $V$ runs over integral substacks of $F$.

To determine $n(W)$ for an integral $W$, pick a closed point $j\colon P \to W$,
so that $j$ is  \'etale locally a regular embedding.
As $c$ commutes with the Gysin pullback $j^!$, we see that $n(W)=n(P)$.
As the relation $c[P]=n(P)[P]$ holds in $B_0(P)$, and $P \to G$ is proper, we can pushforward
$c[P]=n(P)[P]$ to $B_0(G)$. Since $B_0(P) \to B_0(G)$ is injective by Lemma \ref{injectivity}, 
$n(P)$ is determined by  $c\colon B_0(G) \to B_0(G)$. From here it is straightforward to complete the proof.

\begin{lem} \label{injectivity}
Let $i\colon P=\spec \bC \to H $ be a closed point of a proper DM stack $H$, then
both $i_*\colon A_0(P) \to A_0(H)$ and $i_*\colon B_0(P) \to B_0(G)$ are injective.
\end{lem}
\begin{proof}

As $P$ and $H$ are proper, we have a commutative diagram
\[
\xymatrix{
A_0(P) \ar[r]^{i_*}\ar[rd]_{\int_P} & A_0(H)\ar[d]^{\int_H}\\
                             & A_0(\spec \bC)=\bQ.
}\]
Since $\int_P[P]=1$ is nonzero, $i_*$ is injective on $A_0$.
The same argument works for $B_0$.
\end{proof}

\end{proof}

\subsection{}
Now we can prove Theorem \ref{FFF}.
\begin{proof}

For any closed point $l\colon P \to G$ we have a commutative diagram
\[
\xymatrix{
A_0(P) \ar[d]^{l_*} & A_0(P)\ar[d]^{l_*} \ar[l]_c\\
A_0(G)         &   A_0(G) \ar[l]_c.
}
\]

Composing it with
\[
\xymatrix{
A_0(G) \ar[d]^{h_*} & A_0(G)\ar[l]_c\ar[d]^{h_*}\\
A_0(G')          &   A_0(G')\ar[l]_{c'},
}
\]
we have
\[
\xymatrix{
A_0(P) \ar[d]^{l_*} & A_0(P)\ar[d]^{l_*}\ar[l]_c\\
A_0(G) \ar[d]^{h_*} & A_0(G)\ar[d]^{h_*}\\
A_0(G')          &   A_0(G')\ar[l]_-{c'}.
}
\]
If we view $P$ as a point of $G'$ via $h\circ l$, then 
this diagram indicates that $c=c'\colon A_0(P) \to A_0(P).$
Here we used $h_*\circ l_*$ being injective.

\end{proof}


\section{}
We define a compatibility condition between bivariant classes that is weaker than
being pullbacks and prove a lemma that will be used in our proof of Gathmann's theorem.

\begin{defn}
Consider a commutative diagram
\begin{equation} \label{D}
\xymatrix{
F \ar[r]^\mu\ar[d]^f  & F' \ar[d]^{f'}\\
G \ar[r]^\nu          & G',
}
\end{equation}
 between proper DM stacks. Note that $f,f',\mu,\nu$ are all proper maps.

Two bivariant classes $\alpha \in A^{-d}(f)$ and $\alpha' \in A^{-d}(f')$ of the same degree
are compatible (with respect to \eqref{D}) if
\[
\mu_*\circ \alpha=\alpha'\circ \nu_*
\colon A_*(G) \to A_{*+d}(F').
\] \
\end{defn}

The following lemma follows from our definition of compatibility and we omit the proofs.

\begin{lem}  \label{key}
\begin{enumerate}
 \item $f_*\alpha$ and $f'_*\alpha'$ are compatible with respect to
 \[\xymatrix{
 G \ar[r]^\nu\ar[d]^{\on{id}_G}  & G' \ar[d]^{\on{id}_{G'}}\\
 G \ar[r]^\nu          & G',
 }\]
 
 \item Consider a commutative diagram
 \[
 \xymatrix{
F \ar[r]^\mu \ar[d]^f        & F' \ar[d]^{f'}\\
G \ar[r]^\nu \ar[d]^g        & G' \ar[d]^{g'}\\
H  \ar[r]^\eta                  & H'
}
 \]
 If $\alpha \in A^*(f)$ and $\alpha' \in A^*(f')$ are compatible with respect to the upper square, and 
 $\beta \in A^*(g)$ and $\beta' \in A^*(g')$ are compatible with respect to the lower square, 
 then $\alpha\cdot \beta$ and $\alpha'\cdot \beta'$ are compatible with respect to the whole rectangle.

\item \label{hehe}

Let  $g\colon G \to G'$ and $h\colon H \to H'$ be two proper maps, $y$ and $y'$ cycles on $H$ and $H'$ respectively.
Consider the commutative diagram
\[
\xymatrix{
G \times H \ar[r]^{g\times h}\ar[d] &G' \times H'\ar[d]\\
G \ar[r]^g                      & G'.
}
\] where vertical arrows are projections.
Let $\alpha$ and 
$\alpha'$ be given by exterior products $ (-) \times y$ and $(-)\times y'$ respectively, then
$\alpha$ and $\alpha'$ are compatible if $h_*(y)=y'$.

\end{enumerate}
\end{lem}

\section{}
To state Gathmann's theorem, we need some preparation.

Let $X$ be a smooth projective variety over $\bC$, $L$ a line bundle over $X$,
and $\pi \colon Y \to X$ be the projective bundle $\bP_X(L\oplus\cO) \to X$.
See \ref{stablemaps} for our conventions on moduli stacks.

Consider a moduli stack $\mbar_{\Gamma_Y}(Y^\dagger)$ or $\mbar_{\Gamma_Y}(Y^\dagger)^\sim$.
Let $\Gamma_X=\pi_*\Gamma_Y$, 
as long as $\mbar_{\Gamma_X}(X)$ exists,
the map $\pi$ induces a map 
\[
\mbar_{\Gamma_Y}(Y^\dagger) \to \mbar_{\Gamma_X}(X)
\] or 
\[
\mbar_{\Gamma_Y}(Y^\dagger)^\sim \to \mbar_{\Gamma_X}(X).\]

Let $\Gamma_X'$ be obtained from $\Gamma_X$ by forgetting some marked points,
and $\Gamma_X''$ obtained from $\Gamma_X'$ by forgetting some components.
If $\mbar_{\Gamma_X''}(X)$ exists, then we have a forgetful map
\[
\mbar_{\Gamma_X}(X) \to \mbar_{\Gamma_X'}(X)
\] forgetting marked points not in $\Gamma_X'$,
and a projection map
\[
\mbar_{\Gamma_X'}(X) \simeq \mbar_{\Gamma_X''}(X) \times \mbar_{\Gamma_X'- \Gamma_X''}(X) \to \mbar_{\Gamma_X''}(X)\
\]
forgetting components not in $\Gamma_X''$. 
Here $\Gamma_X'- \Gamma_X''$ denotes the data in $\Gamma_X'$ about the components
not in $\Gamma_X''$.

Denote by $p_X$ the composition of 
\[
\mbar_{\Gamma_Y}(Y^\dagger) \to \mbar_{\Gamma_X}(X) \to \mbar_{\Gamma_X'}(X) \to \mbar_{\Gamma_X''}(X)
\]
and $q_X$ the composition of 
\[
\mbar_{\Gamma_Y}(Y^\dagger)^\sim \to \mbar_{\Gamma_X}(X) \to \mbar_{\Gamma_X'}(X) \to \mbar_{\Gamma_X''}(X).
\]

The virtual relative dimension of $p_X$ or $q_X$ is
\[\deg[\mbar_{\Gamma_Y}(Y^\dagger)]^\vir -\deg[\mbar_{\Gamma_X''}(X)]^\vir\]
or \[\deg[\mbar_{\Gamma_Y}(Y^\dagger)^\sim]^\vir -\deg[\mbar_{\Gamma_X''}(X)]^\vir.\]

\begin{thm}[{\cite[Theorem 5.2.7]{Ga}}]\label{bundle}
The map $p_X$ or $q_X$ satisfies the virtual pushforward property if it has
nonnegative virtual relative dimension.
\end{thm}
\begin{proof}
We prove for $p_X$ first.

There exists a closed embedding $i\colon X \to P$ into a homogeneous variety $P$ such that $L$ can be realized as a pullback $i^*N$
for some line bundle $N$ over $P$. Let $Q=\bP_P(N\oplus\cO)$, then
there is a cartesian diagram
\begin{equation}\label{sqr}
\xymatrix{
Y \ar[r]^j\ar[d] & Q\ar[d]\\
X\ar[r]^i          &  P.
}
\end{equation}
Let $\Gamma_Q=j_*\Gamma_Y$, 
      $\Gamma_P=i_*\Gamma_X$,
      $\Gamma_P'=i_*\Gamma_X'$,
      $\Gamma_P''=i_*\Gamma_X''$.
It is easy to check we have cartesian diagrams
\begin{equation} \label{pullback}
\xymatrix{
\mbar_{\Gamma_Y}(Y^\dagger) \ar[r]\ar[d]  & \mbar_{\Gamma_Q}(Q^\dagger)\ar[d]\\
\mbar_{\Gamma_X}(X)               \ar[r]          & \mbar_{\Gamma_P}(P),
}
\end{equation}

\begin{equation}\label{flat}
\xymatrix{
\mbar_{\Gamma_X}(X)               \ar[r]  \ar[d]        & \mbar_{\Gamma_P}(P)\ar[d]\\
\mbar_{\Gamma_X'}(X)               \ar[r]          & \mbar_{\Gamma_P'}(P).
}
\end{equation} \label{prod}
and a commutative diagram
\begin{equation}
\xymatrix{
\mbar_{\Gamma_X'}(X)               \ar[r]  \ar[d]        & \mbar_{\Gamma_P'}(P)\ar[d]\\
\mbar_{\Gamma_X''}(X)               \ar[r]          & \mbar_{\Gamma_P''}(P).
}
\end{equation} 
They can be composed into a commutative diagram
\begin{equation} \label{XXX}
\xymatrix{
\mbar_{\Gamma_Y}(Y^\dagger) \ar[r]\ar[d]^{p_X}  & \mbar_{\Gamma_Q}(Q^\dagger)\ar[d]^{p_P}\\
\mbar_{\Gamma_X''}(X)               \ar[r]          & \mbar_{\Gamma_P''}(P).
}
\end{equation}

We will apply Theorem \ref{FFF} to
\[
\mbar_{\Gamma_X''}(X) \to \mbar_{\Gamma_P''}(P).
\] 
As $P$ is homogenous, $B_0(\mbar_{\Gamma_P''}(P))=\bQ$ by \cite{KP}. 
Using the diagrams above we will construct compatible classes $c$ and $c'$ .

For \eqref{pullback}, there exists a virtual pullback for vertical arrows, it comes from the
log cotangent bundle $T^{\on{log}}_{Q^\dagger/P}$.(cf. \cite[Proposition 4.9]{vpf})
Since $j^*T^{\on{log}}_{Q^\dagger/P}=T^{\on{log}}_{Y^\dagger/X}$ in \eqref{sqr},
this virtual pullback  maps $[\mbar_{\Gamma_X}(X)]^\vir$ to $[\mbar_{\Gamma_Y}(Y^\dagger)]^\vir$.

For \eqref{flat}, we have the flat pullback as a bivariant class for each vertical arrow.
This class maps $[\mbar_{\Gamma_X'}(X)]^\vir$ to $[\mbar_{\Gamma_X}(X)]^\vir$ by
\cite[Axiom IV. forgetting tails]{Be}.

For \eqref{prod}, we use Lemma \ref{key} to construct compatible bivariant classes.
The class for
\[
\mbar_{\Gamma_X'}(X) \simeq \mbar_{\Gamma_X''}(X) \times \mbar_{\Gamma_X'- \Gamma_X''}(X)  \to \mbar_{\Gamma_X''}(X)
\] 
is
$(-) \times [\mbar_{\Gamma'_X-\Gamma_X''}(X)]^{\vir}$, and  it maps
$[\mbar_{\Gamma_X''}(X)]^\vir$ to $[\mbar_{\Gamma_X'}(X)]^\vir$. Here we need
\cite[Axiom II. products]{Be}.

Composing these classes, we get compatible bivariant classes $p_X^! \in A^*(p_X)$ and
$p_P^!\in A^*(p_P)$ by Lemma \ref{key}, and $p^!_X[\mbar_{\Gamma_X''}(X)]^\vir=[\mbar_{\Gamma_Y}(Y^\dagger)]^\vir$.

Let $\gamma_Y \in A^*(\mbar_{\Gamma_Y}(Y^\dagger))$ be a class
whose degree is
$\deg[\mbar_{\Gamma_Y}(Y^\dagger)]^\vir -\deg[\mbar_{\Gamma_X''}(X)]^\vir$,
and it is  made up of
evaluation classes of the form $[Y_0]$ or $[Y_\infty]$, and cotangent line classes.
Note that there is a canonical
$\gamma_Q \in A^*(\mbar_{\Gamma_Q}(Q^\dagger))$ whose pullback to $\mbar_{\Gamma_Y}(Y^\dagger)$
is $\gamma_Y$.

Now we can apply Theorem \ref{FFF} to $(p_X)_*(\gamma_Y \cdot p_X^!)$ and $(p_P)_*(\gamma_Q \cdot p_P^!)$
and conclude the proof for $p_X$. Note that these two classes are compatible by Lemma \ref{key}.

The case for $q_X$ is entirely similar, the only difference is that we will need perfect obstruction theories for rubber moduli stacks
discussed in \cite{MR}.
\end{proof}

\begin{remark}
In case $\Gamma_X''=\Gamma_X'$, we can simply apply \cite[Corollary 3.15]{vpf} since
we have cartesian diagrams, and $p_X^!$ is the pullback of $p_P^!$. 
Also the scalar multiplication for $(X, L)$ can be determined using $(\bP^1, \cO(\deg L))$.

The case when $\Gamma_X''$ differs from $\Gamma_X'$ is only useful for $q_X$.
Relative invariants satisfy the product rule, so we only need to consider connected relative stable maps,
in which case there is no component to forget.
\end{remark}

\begin{remark}
Virtual pushforward property for $p_X$ and $q_X$ can be used to relate invariants of $Y$ to
invariants of $X$, more discussion can be found in \cite[Chapter 5]{Ga}.

Independant of virtual pushforward properties, Maulik and Pandharipande in \cite{MP} showed that all 
(desendant) invariants of $Y$ can be effectively reconstructed from invariants of X. 
For the genus zero case, 
combining virtual pushforward properties and techniques in \cite{MP},
(ancestor) invariants of $(Y, Y_0\coprod Y_\infty)$ can be determined by $X$ in a 
simple manner. 
(See \cite[Section 3]{LLQW}.)

\end{remark}

\end{document}